\documentclass[12pt]{amsart}
\usepackage[T1]{fontenc}
\usepackage[utf8]{inputenc}
\usepackage[width=155mm,top=20mm,bottom=20mm,headheight=15pt]{geometry}
\usepackage{amsmath,amssymb,amsthm}
\usepackage{graphicx}
\usepackage{tikz-cd}
\usepackage{multicol}
\usepackage[all]{xy}
\usepackage{xcolor}
\usepackage[rightcaption]{sidecap}
\usepackage{hyperref}
\usepackage{caption, subcaption} %figuras y subfiguras
\usepackage{array, tabularx, multirow} %herramientas para tablas

\newtheorem{teo}{Theorem}[section]

\newtheorem{lema}{Lemma}[section]

\newtheorem{prop}{Proposition}[section]
\newtheorem{proposition}[prop]{Proposition}
\newtheorem{cor}{Corollary}[section]
\newtheorem{corollary}[cor]{Corollary}
\newtheorem{defi}{Definition}[section]

\newtheorem{remark}{Remark}[section]
\newtheorem{conj}{Conjecture}[section]

\newenvironment{claim}[1]{\par\noindent\underline{Claim:}\space#1}{}
\newenvironment{claimproof}[1]{\par\noindent \textit{Proof of Claim:}\space#1}{\hfill $\square$}

\newcommand{\OO}{\mathcal{O}}

\newcommand{\Aab}{\mathcal{A}}
\newcommand{\grd}{\mathfrak{g}^r_d}
\newcommand{\mvl}{M_{V,L}}

\begin{document}

\title{An answer for a Mistretta-Stoppino's conjecture}

\author{Erick David Luna Núñez}

\address{Departamento de Matemáticas y Física, Universidad Autónoma de Aguascalientes
\newline  Av. Universidad \#940, Ciudad Universitaria \newline C.P. 20100, Aguascalientes, Ags. México.}

\email{erick.luna@edu.uaa.mx}

\date{}

\begin{abstract}
We study the relation between linear stability of generated linear series on smooth curves and slope stability of their associated syzygy bundles. Motivated by conjectures of Mistretta and Stoppino, we establish new cases in which linear stability implies slope stability, focusing first on generated linear series over general curves and then on curves lying on polarized K3 surfaces. In the case of general curves, we use Brill-Noether-theoretic arguments to relate the numerical conditions on the linear series to the semi-stability of the syzygy bundle. For curves on K3 surfaces, we combine Lazarsfeld-Mukai bundles with Bridgeland stability conditions and restriction techniques to obtain slope-stability results under explicit degree bounds. These results provide further evidence for the expected equivalence between linear stability of linear series and slope stability of syzygy bundles.
\end{abstract}

\maketitle

\section{Introduction}

Let $C$ be a smooth projective curve over the complex numbers and let $L$ be a globally generated line bundle of degree $d$ on $C$. Consider a subspace $V \subset H^0(L)$ of dimension $r+1$ that generates $L$. The pair $(L,V)$ is called a generated linear series of type $(d,r+1)$. Associated to this linear series is the syzygy bundle $\mvl$, defined as the kernel of the evaluation map $V \otimes \OO_C \xrightarrow{ev} L$, so that we have an exact sequence
\[
0 \rightarrow \mvl \rightarrow V \otimes \OO_C \xrightarrow{ev} L \rightarrow 0.
\]

The bundle $\mvl$ is also known as the syzygy bundle, kernel bundle, or dual span bundle (DSB). When $V=H^0(L)$, we denote $M_{H^0(L),L}$ by $M_L$. The slope stability of $\mvl$ is closely connected to the geometry of Brill-Noether loci and to the minimal resolution conjecture; see, for instance, \cite{farkas2003divisors}. In \cite{stoppino2008slope}, Stoppino extends Mumford's notion of linear stability for projective varieties \cite{mumford1977stability} to generated linear series on curves.

A central problem is to determine when linear semistability of $(L,V)$ implies slope semistability of the associated syzygy bundle $\mvl$. This question is addressed by two conjectures of Mistretta and Stoppino \cite{mistretta2012linear}:

\begin{conj}\label{conj1}
    Let $C$ be a smooth curve of genus $g$ and $(L,V)$ be a generated linear series of type $(d,r+1)$ over $C$ with $V \subsetneq H^0(L)$. If $d \leq kr$ where $k$ denotes the gonality of $C$, then linear (semi)stability of $(L,V)$ is equivalent to slope-(semi)stability of $\mvl$.
\end{conj}

\begin{conj}\label{conj2}
    For any curve $C$, and any line bundle $L$ on $C$, linear (semi)stability of $(L,H^0(L))$ is equivalent to slope-(semi)stability of $M_L$.
\end{conj}

Conjecture \ref{conj2} is known to hold for general and hyperelliptic curves by \cite{castorena2018linear}. Our main goal is to provide new families of curves for which Conjecture \ref{conj1} is valid. Using techniques from \cite{castorena2018linear}, together with the fact that for a general curve the inequality $d \leq g+r$ implies $d \leq kr$, we obtain the following statement.

\begin{corollary}\label{TeoA}
Let $(L,V)$ be a generated linear series of type $(d,r+1)$ over a general curve $C$ of genus $g \geq 2$ with $codim_{H^0(L)}(V) \leq h^1(L)$. Then linear (semi)stability of $(L,V)$ is equivalent to the slope-(semi)stability of $\mvl$.
\end{corollary}

We also study curves lying on principally polarized K3 surfaces. More precisely, if $(X,H)$ is a principally polarized K3 surface and $C \in |H|$, we analyze the Lazarsfeld-Mukai bundle $F_{V,L}$ associated to $(L,V)$. By studying stability conditions on $X$, we prove that $F_{V,L}$ is $H$-Gieseker stable and, under an additional degree bound, that its restriction to $C$ is slope-stable. This leads to the following result.

\begin{proposition}\label{intro:equivk3surf}
    Let $(X,H)$ be a polarized K3 surface with the property that $H^2$ divides $H.D$ for any curve classes $D$ on $X$. Take $C \in |H|$ a curve of genus $g>2$, let $(L,V)$ be a generated linear series of type $(d,r+1)$ over $C$ with $1<r<d \leq \min\{g-1,kr\}$, where $k$ is the gonality of $C$. If $(L,V)$ is linearly stable, then $\mvl$ is slope-stable.
\end{proposition}

\section{Determinant bundles} 

In this section, we present similar results to those given by Castorena and Torres-Lopez in \cite{castorena2018linear} emphasizing the diferences with our case. As a first approach to understand the Conjecture \ref{conj1}, we have the following result of \cite{mistretta2012linear} giving conditions for the bundle $F_S$ and for its determinant bundle $det(F_S)$. 

\begin{lema}[ \text{\cite[Lemma 4.3]{mistretta2012linear}} ] \label{lemamist}
Let $C$ be a $k$-gonal curve. With the notation of Butler's diagram of the pair $(L,V)$ by $S$, suppose that $rank(F_S) \geq 2$. If $F_S$ fits in an exact sequence $$0 \rightarrow \bigoplus^{rank(F_S)-1} \OO_C \rightarrow F_S \rightarrow det(F_S) \rightarrow 0$$ which is also exact on global sections, then the following properties hold: 
\begin{enumerate}
    \item If $deg(L)\leq k (dim(V)-1)$ where $k=gon(C)$, then $\mu (S) \leq \mu (\mvl)$. Furthermore, we have equality if and only if
    \begin{itemize}
        \item $W=H^0 (F_S)$
        \item $k = \frac{deg(det(F_S))}{h^0(det(F_S))-1}$
        \item $k = \frac{deg(L)}{dim(V)-1}$
    \end{itemize}
    \item If $deg(L) < k (dim(V)-1)$ where $k=gon(C)$, then $\mu (S) < \mu (\mvl)$
\end{enumerate}
\end{lema}

By dualizing the first exact row in Butler's diagram of the pair $(L,V)$ by $S$ ($0 \rightarrow S \rightarrow W \otimes \OO_C \rightarrow F_S \rightarrow 0$) and twisting by $K_C$, we get a map at sections
$$m_W : W^\vee \otimes H^0 (K_C) \rightarrow H^0 (S^\vee \otimes K_C).$$ In contrast to \cite{castorena2018linear} we consider the incomplete case, that is $(L,V)$ a generated linear series with $V \in Gr(r+1,H^0 (L))$ generating space with $r+1<h^0 (L)$ and the Butler's diagram of the pair $(L,V)$ by $S$. 

\begin{prop}
Let $Q = \mvl/S$ 
\begin{itemize}
    \item[i)] If the multiplication map $$m_W : W^\vee \otimes H^0 (K_C) \rightarrow H^0 (S^\vee \otimes K_C)$$ is surjective, then $H^0 (Q)=0$.
    \item[ii)] If $m_W$ is surjective, then $W=H^0 (F_S)$.
    \item[iii)] If $S\subset \mvl$ is stable of maximal slope, then $H^0 (Q)=0$.
\end{itemize}
\label{Prop1}
\end{prop}

\begin{proof}  
\begin{itemize}
    \item[i)] We follow \cite{castorena2018linear}. By dualizing Butler's diagram of $(L,V)$ by $S$ as shown below, twisting by $K_C$ and taking cohomology,
    \[
    \resizebox{\linewidth}{!}{$
    \xymatrix{
    H^0 (L^\vee \otimes K_C) \ar[r] & V^\vee \otimes H^0(K_C) \ar[r]^{m_1} \ar[d]^{p_1} & H^0(\mvl^\vee \otimes K_C) \ar[d]^{p_2} \ar[r]^{A} & H^1(L^{\vee} \otimes K_{C}) \ar[r]^{D} & V^\vee \otimes H^{1} (K_{C}) \ar[r] & 0 \\
    H^{0} (F_{S}^{\vee} \otimes K_{C}) \ar[r] & W^\vee \otimes H^{0} (K_{C}) \ar[r]^{m_W} & H^0(S^{\vee} \otimes K_{C}) \ar[r]^{E} \ar[d]^{a} & H^1(F_{S}^{\vee} \otimes K_{C}) \ar[r]^{B} & W^{\vee} \otimes H^{1} (K_{C}) \ar[r] & 0 \\
    & & H^1(Q^{\vee} \otimes K_{C}) \ar[d]^{b} & & & \\
    & & H^1(\mvl^{\vee} \otimes K_{C}) & & &
    }
    $}
    \]

    Since $ m_W $ is surjective and $W \hookrightarrow V $, it follows that $ p_1 $ is also surjective, and $ m_W \circ p_1 $ is surjective as well. By commutativity of the diagram $ m_W \circ p_1 = p_2 \circ m_1 $ is surjective and this implies that $ p_2 $ is also surjective, which is equivalent to stating that the morphism $ a $ is equal to zero.

    On the other hand, by Serre duality,    
    $$ H^1(Q^\vee \otimes K_C) \cong H^0(Q)^\vee \quad \text{ and } \quad H^1(\mvl^\vee \otimes K_C) \cong H^0(\mvl)^\vee. $$ 
    Furthermore, since  $H^0(\mvl) \cong 0 $ then $ a \equiv 0 $ and $ b $ is an isomorphism. We conclude that $ H^0(Q)^\vee \cong 0 $, equivalently $ H^0(Q) = 0 $.
    
    \item[ii)] If $ m_W $ is surjective then the map $ E $ vanishes and $ B $ is an isomorphism. By Serre duality $H^1(F_S^\vee \otimes K_C) \cong H^0(F_S)^\vee. $     Thus $ H^0(F_S)^\vee \cong W^\vee \otimes H^1(K_C),$ leading to the isomorphisms,
        \begin{equation*}
        H^0(F_S)^\vee \cong W^\vee \otimes H^1(K_C) \cong W^\vee. 
    \end{equation*}
      By dualizing we obtain $ H^0(F_S) \cong W $.
    
    \item[iii)] This proof is analogous to \cite[Theorem 1.1]{castorena2018linear} which is based on the study of $ Q $ when $\mvl$ is slope-semistable (for which $ Q $ is slope-semistable of negative degree) or when $\mvl$ is slope-unstable (in this case, the maximal slope of the Harder-Narasimhan filtration for $ Q $ is negative, leading to $ H^0(Q) = 0 $).
    \end{itemize}
\end{proof}

\begin{remark} \label{obs1}
    In Theorem 1.1 of \cite{castorena2018linear} the case \textit{i)} is an \textbf{if and only if} result. However, in the non-complete case as in Proposition \ref{Prop1}, the proof of $H^0 (Q)=0$ implies that the map $m_W$ is surjective is non true in general. This is due to the fact that the morphism $D: H^1 (L^\vee \otimes K_C) \rightarrow V^\vee \otimes H^1 (K_C)$ (following the notation of the proof) is not an isomorphism since $r+1<h^0 (L)$. Consequently, when $V \subsetneq H^0 (L)$ we cannot assert that the morphisms $m_1$ and $p_2$ are surjective as the authors do in \cite{castorena2018linear}. This indicates that the results established for the complete case cannot be directly extended to the incomplete case.
\end{remark}

Using this construction, for a general curve we can give a proof of the slope-semistability of $\mvl$ and conditions for the strictly slope-semistability of $\mvl$.

\begin{prop} \label{Prop2}
Let $C$ be a general curve of genus $g \geq 2$. Let $(L,V)$ be a generated linear series of type $(d,r+1)$ on $C$, and consider $c \colon =codim_{H^0 (L)} V \leq h^1(L)$. Then $\mvl$ is slope-semistable. Moreover, if there exists a proper subbundle $S \subset \mvl$ with $\mu (S) = \mu (\mvl)$, then
\begin{itemize}
    \item $h^1 (L)-c=0$.
    \item $s:=rank(S)=r-1$.
    \item $d=g+r$ with $r|g$.
\end{itemize}
\end{prop}

\begin{proof}
This proof is analogous to \cite[Lemma 4.1]{castorena2018linear}; we will refer to this proof later but for completeness we will provide the entire argument here. Consider a proper subbundle $S \subset \mvl$ with inclusion $0 \rightarrow S \rightarrow \mvl$. By dualizing the sequence $\mvl^\vee \rightarrow S^\vee \rightarrow 0$, we obtain $S^\vee$ as a quotient of $\mvl^\vee$. Consequently, for any $U \in Gr(s+1,H^0 (S^\vee))$, we have the following exact sequence,
$$0 \rightarrow S \rightarrow U^\vee \otimes \OO_C \rightarrow det(S^\vee) \rightarrow 0$$ inducing the exact sequence in cohomology $$0 \rightarrow H^0 (S) \rightarrow U^\vee \rightarrow H^0 (det(S^\vee)) \rightarrow ...$$
Since $h^0 (\mvl)=0$ and $S \hookrightarrow \mvl$, it follows that $H^0 (S)=0$, leading to $h^0 (det(S^\vee)) \geq dim (U)=s+1$. Using that $C$ is general we get $h^0 (det(S^\vee)) \geq s+1$ and $deg(det(S^\vee))=deg(S^\vee)$, the last argue implies that $det(S^\vee)$ is a degree $deg(S^\vee)$ line bundle with at least $s+1$ global sections. The Brill-Noether number for $det(S^\vee)$ is given by $$\rho (g,s,deg(S^\vee)) = g-(s+1)(g-deg(S^\vee)+s) \geq 0$$ 
where
\begin{equation}
    deg(S^\vee ) \geq \frac{s(s+g+1)}{s+1}. \label{conditsprima}
\end{equation}
Thus,
$$\mu (S) = \frac{-deg(S^\vee)}{s} \leq - \frac{\frac{s(s+g+1)}{s+1}}{s} = -\frac{s+g+1}{s+1}=-1-\frac{g}{s+1}$$ 

Now, from Riemann-Roch theorem for the line bundle $L$ and letting $h=h^1 (L)$, we have 
$$\mu(\mvl)=-\frac{d}{r} = \frac{r+h^0(L)-(r+1)-h+g}{r} = \frac{r+c-h+g}{r}=-1+\frac{h-c-g}{r}.$$ 
Consequently,  
\begin{equation} \label{ineq:semistable-bound}
    \mu (S) - \mu (\mvl) \leq -1-\frac{g}{s+1} + 1 - \frac{h-c-g}{r} = g \left( \frac{1}{r} - \frac{1}{s+1} \right)-\frac{h-c}{r}.
\end{equation}
Since $h-c$ is greater than $0$, inequality \ref{ineq:semistable-bound} is less or equal to $0$ which implies that $\mu(S)\leq \mu(\mvl)$. Therefore, we conclude that $\mvl$ semistable. 
For the last inequality in \ref{ineq:semistable-bound}, if $\mu (S) = \mu (\mvl)$, then $h^1=c$ and $r=s+1$. Moreover, from Riemann-Roch theorem we get that $d=g+r$ and $\mu(\mvl)=deg(\mvl/S) \in \mathbb{Z}$ and we deduce that $r$ divides $g$. 
\end{proof}

At this point, we turn our attention to the bundle $F_S$ associated with the subbundle $S$ rather than $S^\vee$, as $F_S$ appears in the Butler's diagram of $(L,V)$ by $S$. To effectively apply Lemma \ref{lemamist}, we consider the case in which $\mvl$ is strictly slope-semistable, with $S$ being a subbundle that shares the same slope as $\mvl$, specifically $\mu(S) = \mu(\mvl)$. This condition allows us to compute the dimension of global sections for the bundle $det(F_S)$, which for the case of $rank(F_S) \geq 2$ will provide valuable insights into the stability properties of $\mvl$. 

\begin{proposition}\label{Prop3}
Let $C$ be a general curve of genus $g \geq 2$, $(L,V)$ be a generated linear series over $C$ and let $c=codim_{H^0 (L)} (V) \leq h^1(L)$. If $S$ is a proper subbundle of $\mvl$ with $\mu (S)=\mu (\mvl)$. Then
$$h^0 (det(F_S)) = s+1 .$$
\end{proposition}

\begin{proof}
As in \cite{castorena2018linear} we have that $$\mu(S)=-\frac{deg(F_S)}{s}=\frac{deg(S)}{s}=-\frac{deg(L)}{r}=\mu(\mvl).$$ From Proposition \ref{Prop2} then $deg(F_S) = s+g-\frac{g}{r} \in \mathbb{Z}$. Notice that $h^0 (det(F_S)) = h^0 (det(S^\vee)) \geq s+1 = r$. Assume that $h^0 (det(F_S))\geq r+1$. Using that $C$ is general and $det(F_S)$ is a degree $deg(F_S)$ line bundle with at least $r+1$ sections, it follows that the corresponding Brill-Noether number $\rho=\rho(g,r,deg(F_S))$ is non-negative. However, we have,
$$0 \leq \rho = g-(r+1)(r-deg(F_S)+g)=g-(r+1)\left (1+\frac{g}{r} \right) = -r-1-\frac{g}{r}<0$$ 
contradicting the generality of $C$. Therefore $h^0 (det(F)) = s+1=r$.
\end{proof}

We aim to prove that under the conditions outlined in Proposition \ref{Prop2} we can conclude the same result as in part \textit{ii)} of Proposition \ref{Prop1}. From Remark \ref{obs1}, this conclusion does not hold in general for the non-complete case. By establishing the semistability of $\mvl$, we can analyze the properties of the bundle $F_S$ associated with the subbundle $S$. This connection is essential, as it allows us to leverage the results from the semistability of $\mvl$ to draw conclusions about the behavior of $F_S$.

\begin{proposition}
    Let $C$ be a general curve of genus $g \geq 2$, $(L,V)$ be a generated $\grd$ over $C$, if $S$ is a proper subbundle of $\mvl$ with $\mu(S) = \mu (\mvl)$ and $rank(F_S)>1$, then $W=H^0 (F_S)$. \label{proph0}
\end{proposition}

\begin{proof}
    Since $\mvl$ is strictly slope-semistable from Proposition \ref{Prop2} we conclude that $s=rank(S)=r-1$. From the inclusion $W \hookrightarrow V$ the dimension $w=dim(W)$ can take only two possible values: either $w = r+1$ or $w=r$. First, consider the case where $w=r$. In this case we have $rank(F_S) = 1$, which does not fall within the cases we are currently considering. Next, we assume $w=r+1$. In this case $W \cong V$ leading to $rank(F_S) = 2$. Since $W \subset H^0(F_S)$, it follows that $w \leq h^0 (F_S)$. 
     Moreover, according to \cite[Proposition 4.1]{mistretta2012linear}, $F_S$ fits into the following exact sequence $$0 \rightarrow \OO_C \rightarrow F_S \rightarrow det(F_S) \rightarrow 0.$$ 
     By taking cohomology  
    \begin{equation}
        0 \rightarrow H^0 (\OO_C) \rightarrow H^0 (F_S) \xrightarrow{\varphi } H^0 (det(F_S)) \rightarrow \cdots.
        \label{suc:fs-sections}
    \end{equation}
    
    From Proposition \ref{Prop3} we get $h^0 (det(F_S))=s+1=r$ and from dimension theorem, 
    $$h^0 (F_S) = dim(Im(\varphi)) + dim(Ker (\varphi)).$$
    
    Using the exactness of the sequence \ref{suc:fs-sections} and the fact that $Im(\varphi)$ is a subspace of $H^0 (det(F_S))$, we obtain the following inequality
    $$h^0 (F_S) = dim(Im(\varphi)) + dim(Ker(\varphi)) = dim(Im(\varphi))+1 \leq r+1.$$     
    We conclude that $r+1=w \leq h^0 (F_S) \leq r+1$, which implies that $W = H^0 (F_S)$.
\end{proof}

We recall that linear semistability of $(L,V)$ is equivalent to the fact that the bundle $\mvl$ cannot be destabilised by subbundles of the form $M_{W,L'}$. Next we show that it is possible to construct a specific linear subseries associated with a proper subbundle $S \subset \mvl$ that shares the same slope. 

\begin{teo}
\label{teo 1}
Let $C$ be a general curve of genus $g\geq 2$ and let $(L,V)$ a globally generated linear series such that $dim(V)=r+1$ with $c =codim_{H^0 (L)} (V) \leq h^1(L)$. Consider $S\subset \mvl$ a proper subbundle with $\mu(S)=\mu(\mvl)$. Then, there exists a line bundle $F_S$ which fits into the commutative diagram
$$
\xymatrix{0 \ar[r] & S \ar[r] \ar[d] & W\otimes \OO_C \ar[r] \ar[d] & F_S \ar[r] \ar[d] & 0 
\\ 
0 \ar[r] & \mvl \ar[r] & V \otimes \OO_C \ar[r] & L \ar[r] & 0
}
$$
\end{teo}
\begin{proof}
This proof is analogous to \cite[Theorem 4.1]{castorena2018linear}. 
\end{proof}

From above results, in the context of Theorem \ref{equivgenericcase}, we know properties of the elements that appear in the Butler's diagram of $(L,V)$ by $S$ when $S$ has the same slope as $\mvl$. This allows us to leverage these properties to conclude that the linear stability of $(L,V)$ implies the slope-stability of $\mvl$.

\begin{cor} \label{equivgenericcase}
Let $(L,V)$ be a generated linear series of type $(d,r+1)$ over a general curve $C$ of genus $g \geq 2$ with $c\leq h^1(L)$. Then linear (semi)stability of $(L,V)$ is equivalent to slope-(semi)stability of $\mvl$
\end{cor}

\begin{proof}
    If $\mvl$ is stable, then $(L,V)$ is linearly stable. Suppose that $\mvl$ is strictly slope-semistable and let $S\subset \mvl$ be a subbundle with $\mu (S)= \mu (\mvl)$. By Theorem \ref{teo 1} the rank of $F_S$ is equal to $1$ and this implies that $(L,V)$ is strictly linear semistable.
\end{proof}

Finally, with the same assumptions as in the Corollary, we aim to study under which conditions $\mvl$ fails to be slope-stable in terms of the elements of the linear series $(L,V)$. Specifically, we explore how the properties of the elements in the Butler's diagram of $(L,V)$ by $S$ when $S$ has the same slope as $\mvl$ influence the slope-stability of the vector bundle $\mvl$. The proof is similar to that given in Castorena and Torres-Lopez with the difference that in our case we have to consider the linear series $(L(-Z),V(-Z))$ where the space $V(-Z):=H^0(L(-Z)) \cap V$ for a specific effective divisor $Z$ on $C$.  

\begin{prop} \label{condest}
    Let $C$ be a general curve of genus $g \geq 2$ and let $(L,V)$ be a globally generated linear series such that $dim(V)=r+1$ and $deg(L)=d$. Then, $\mvl$ fails to be stable if and only if the following three conditions hold:
    \begin{itemize}
        \item[$i)$] $c:=codim_{H^0(L)}(V)=h^1(L)=h^1$.
        \item[$ii)$] $d=g+r$ with $r|g$.
        \item[$iii)$] There is an effective divisor $Z$, with $h^0 (L(-Z))=h^0(L)-1$ and $dim(V(-Z)) = r$, where $V(-Z)=V \cap H^0(L(-Z)) \subset H^0(L)$ and $deg(Z)=1+\frac{g}{r}$.
    \end{itemize}
\end{prop}

\begin{proof}
    \begin{itemize}
        \item[$\Leftarrow )$] The evaluation map $ev:H^0(L(-Z)) \otimes \OO \rightarrow L(-Z)$ is surjective. Notice that $V \subsetneq H^0(L(-Z))$ because $dim(V \cap H^0 (L(-Z)))<dim(V)$. We know that $L(-Z)$ is generated by $V(-Z)$; otherwise, there exists a point $p \in C$ such that $ev|_p : V(-Z) \otimes \OO_p \rightarrow L(-Z)_p$ is not surjective. This leads to two cases: 
        \begin{itemize}
            \item If $p \in C$ and $p \not \in Z$, since $V$ generates $L$, the evaluation map $ev|_p : V(-Z) \otimes \OO_p \rightarrow L(-Z)_p$ is surjective.
            \item If $p \in Z$, given that the map $ev:H^0(L(-Z)) \otimes \OO \rightarrow L(-Z)$ is surjective, it follows that $ev|_p : V(-Z) \otimes \OO_p \rightarrow L(-Z)_p$ is also surjective.
        \end{itemize}
        
        Moreover, we have $F_S=L(-Z)$ and $W=V(-Z)=H^0(L(-Z)) \cap V$ in the Butler's diagram of $(L,V)$ by $S=M_{V(-Z),L(-Z)}$, with $\mu (M_{V(-Z),L(-Z)}) = \mu (\mvl)$. Hence, $\mvl$ is slope-semistable but not slope-stable.
        \item[$\Rightarrow )$] Following the ideas in \cite[Corollary 4.3]{castorena2018linear}. If $\mvl$ is strictly slope-semistable. According to Proposition \ref{Prop2}, there exists a subbundle $S\subset \mvl$ of rank $r-1$ such that $c=h^1(L)$ and $d=g+r$ with $r|g$. 
        
        Since $F$ is a line bundle and the morphism $\alpha : F \rightarrow L$ is non-zero, there exists an effective divisor $Z$ such that $F=L(-Z)$ and $dim(W)=r$. We have the inclusion $W \hookrightarrow V$, since $F$ is generated $W \hookrightarrow H^0 (L(-Z))$ making $W$ a subspace of $H^0 (L(-Z)) \cap V=V(-Z)$ of maximal dimension, thus $W=V(-Z)$. 
        
        Since $deg(F)=-deg(S)$ and $\mu(S)=\mu(\mvl)$, it follows that $deg(Z)=1+\frac{g}{r}$. Note that $Q:=\mvl/S = \OO_C (-Z)$, so $h^0 (Q)=0$ and $\dim (H^0 (L(-Z)) \cap V)= \dim W=r$, which gives condition $iii)$. 
    \end{itemize}
\end{proof}

\subsection{Cohomological stability}

In this section, we consider another type of stability for vector bundles. In \cite{mistretta2012linear} authors studied the relation between this type of stability (that results stronger than the slope-stability) for $\mvl$, the slope-stability of $\mvl$ and linear stability of $(L,V)$.
We present similar results of to those presented by Castorena and Torres-Lopez in \cite{castorena2021new}, we show sketches of the proofs where we emphatize the diferences with our case.

\begin{defi} \label{cohstab}
    Let $E$ be a vector bundle over a curve $C$. We say that $E$ is cohomologically (semi)stable if for any $A \in Pic^a (C)$, and for every $t < rank(E)$, we have that $$h^0 \left( \bigwedge^t E \otimes A  \right) =0$$ whenever $a\leq t \cdot \mu(E)$ ( \text{respectively } $a < t \cdot \mu(E)$).
\end{defi}

In \cite{ein1992stability}, authors show that  cohomological semistability is equivalent to slope-semistability, and that cohomological stability implies slope-stability. We want to look for precise conditions for when such stabilities are equivalent since we know they satisfy the following implications:
\[
\begin{gathered}
\text{Cohomological stability of }\mvl
\Rightarrow \text{ slope stability of }\mvl,\\
\text{Slope stability of }\mvl
\Rightarrow \text{ linear stability of }(L,V).
\end{gathered}
\]

In order to find precise conditions for the equivalence between the first two stabilities, we study the following property for syzygy bundles associated to linear subseries related by divisors.

\begin{lema}[\text{\cite[Lemma 7.4]{mistretta2012linear}}] \label{Lem2}
    Let $(L,V)$ be a $\grd$ on a smooth curve $C$, which induces a birational morphism, and let $D_k = p_1 + ... + p_k$ be a general effective divisor on $C$, with $k<r$. The kernel bundle associated to the linear series lies in the following exact sequence of sheaves $$0 \rightarrow M_{V(-D_k), L(-D_k)} \rightarrow \mvl \rightarrow \bigoplus^k_{i=1} \OO_C (-p_i) \rightarrow 0.$$
\end{lema}

\begin{remark} \label{rmk1}
\begin{itemize}
    \item With notation and conditions in Lemma \ref{Lem2}, if we consider a general effective divisor $D$ of maximal degree $r-1$, we have that $M_{V(-D), L(-D)}$ is a line bundle which is dual to $\OO_C(p_r+...+p_d)$ and $$M_{V(-D), L(-D)} \cong \OO_C (-p_r - ... - p_d).$$
    \item Let $(L,V)$ and $x_1,...,x_{r-1}$ be as in Lemma \ref{Lem2} and let $F = \oplus^{r-1}_{j=1} \OO_C (-x_j)$. For any integer $t<r$, we get the following short exact sequence of exterior powers 
    \begin{equation} \label{suc1}
        0 \rightarrow \bigwedge^{t-1} F \otimes L^\vee \left( \sum^{r-1}_{j=1} x_j \right) \rightarrow \bigwedge^t \mvl \rightarrow \bigwedge^t F \rightarrow 0.
    \end{equation}
\end{itemize}
\end{remark}

Now, we follow \cite{castorena2021new} in the incomplete case in direction to give conditions for which bundle $\mvl$ is cohomological (semi)stable.

\begin{prop} \label{prop5}
Let $(L,V)$ be a generated $\grd$ on a smooth curve $C$ which induces a birational morphism. Let $A \in Pic^a(C)$ such that $a\leq t\cdot \frac{d}{r}$ and $h^0 (A) \leq t$ with integers $t,d$ and $r$ satisfying $0<t<r<d$. Then $h^0 (\wedge^t \mvl \otimes A) = 0$.
\end{prop}

\begin{proof}
    This proof is analogous to \cite[Proposition 3.4]{castorena2021new}.

\end{proof}

From now on, let us assume that $C$ is a general curve. We aim to establish a bound on the dimension of $H^0(A)$ for a line bundle $A$, as above.

\begin{remark}
    It follows from Riemman-Roch formula that hypothesis $$c:=codim_{H^0 (L)} (V) \leq h^1(L)$$ is equivalent to condition $d\leq g+r$.
\end{remark}

\begin{prop} \label{prop6}
    Let $C$ be a general curve of genus $g$. Let $A \in Pic^a (C)$ with $a \leq t \cdot \frac{d}{r}$ with integers $t,d$ and $r$ satisfying $0<t<r<d \leq g+r$, then $h^0 (A) \leq t+1$. Moreover, $h^0(A) = t+1$ if and only if $a=t\cdot \frac{d}{r}$, $d=g+r$ and $t+1=r$.
\end{prop}

\begin{proof}
    This proof is analogous to \cite[Proposition 3.6]{castorena2021new}.
\end{proof}

To establish the cohomological semistability of $\mvl$, notice that from Proposition \ref{Prop2} and the results presented in \cite{ein1992stability}, cohomological semistability is equivalent to slope-semistability. This leads us to conclude that $\mvl$ is cohomologically semistable. We aim to characterize the conditions under which this cohomological semistability holds. Specifically, we analyze the implications of the Brill-Noether numbers associated with $A$, as well as the dimensions of the relevant cohomology groups, to provide a comprehensive understanding of the stability properties of the sheaf.

\begin{prop} \label{teo3}
    Let $(L,V)$ be a generated linear series of type $(d,r+1)$ which induces a birational morphism over a smooth general curve $C$. Then $\mvl$ is cohomologically semistable.
\end{prop}

\begin{proof}
    This proof is analogous to \cite[Theorem 3.7]{castorena2021new}, let $t<r$ and consider $A \in Pic^a (C)$ with $a<t \frac{d}{r}$. From Propositions \ref{prop5} and \ref{prop6} $h^0 (A) \leq t$ and $h^0 (\wedge^t \mvl \otimes A) = 0$. Hence $\mvl$ is cohomologically semistable.
\end{proof}

As for slope-semistability, we want to find suitable conditions for the cohomological stability of $\mvl$. A first step in this direction is the following result.

\begin{cor} \label{cor2}
    Let $(L,V)$ be a generated linear series of type $(d,r+1)$ which induces a birational morphism over a general curve $C$ with $c \leq h^1 (L)$. Then $\mvl$ is cohomologically stable if one of the following conditions holds:
    \begin{itemize}
        \item[$i)$] $c<h^1 (L)$.
        \item[$ii)$] $c=h^1 (L)$ and $r$ does not divide g.
    \end{itemize}
\end{cor}

\begin{proof}
    In a similar way as in \cite[Corollary 3.8]{castorena2021new} for $t<r$ and $A \in Pic^a(C)$ with $a \leq \frac{d}{r}t$. In case $i)$, since $c<h^1 (L)$ then $d<g+r$. From Proposition \ref{prop6} $h^0 (A) \leq t$ and case $i)$ follows from Proposition \ref{prop5}. 
    
    If $h^1(L)=c$ and $t=r-1$, then $d=g+r$ and $$t \frac{d}{r} = (r-1)\frac{g+r}{r} = g+r-1-\frac{g}{r}.$$ Assume that $r$ does not divide $g$ then the condition $a \leq (r-1)\frac{d}{r}$ implies $a<(r-1)\frac{d}{r}$ and hence $h^0 (A) \leq t=r-1$. This proves $ii)$.
\end{proof}

Next result states that the slope-stability of $\mvl$ is equivalent to its cohomological stability. To achieve this, we generalize the results presented in \cite{castorena2021new} to the setting of the sheaf $\mvl$. We show that the conditions under which $\mvl$ is slope-stable imply its cohomological stability, and viceversa. This equivalence allows us to leverage the powerful tools of Brill-Noether theory and the study of Brill-Noether varieties to draw conclusions about the stability properties of $\mvl$.

\begin{teo} \label{Teocohst}
    Let $(L,V)$ be a generated $\grd$ over a general curve $C$ which induces a birational morphism  with $c\leq h^1(L)$. Then
    \begin{enumerate}
        \item $\mvl$ is strictly slope-semistable if and only if the following three conditions hold:
        \begin{enumerate}
            \item $h^1(L)=c$.
            \item $d=g+r$ and $r|g$.
            \item There is a line bundle $A$ with degree $deg(A)=g+r-1-\frac{g}{r}$ and $h^0(A)=r$ such that $h^0(\wedge^{r-1} \mvl \otimes A) =1$.
        \end{enumerate}
        \item If $\mvl$ is slope-stable, then $\mvl$ is cohomologically stable.
    \end{enumerate}
\end{teo}

\begin{proof}
    The proof is the same as in \cite[Corollary 3.10]{castorena2021new}.
\end{proof}

It follows from Proposition \ref{Prop2} and Theorem \ref{cohstab} that for general curves satisfying assumptions of Theorem \ref{cohstab}, the linear stability of the pair $(L,V)$ is equivalent to the slope-stability of $\mvl$, and this slope-stability is equivalent to the cohomological stability of $\mvl$.

\section{Stability conditions}

In this chapter, we examine the fundamental concepts of stability conditions on the bounded derived category $D^b (X)$ of coherent sheaves on a smooth projective variety $X$. The translation functor on $D^b (X)$ denoted by $[1]$.

A key ingredient in our analysis is the construction of a certain abelian subcategory $Coh^\beta (X)$ within $D^b (X)$, which consists of two-term complexes. This abelian category $Coh^\beta (X)$ depends on the choice of a real parameter $\beta$.

To establish the necessary foundations for working with $Coh^\beta (X)$, first we need to introduce some basic notions from homological algebra and the theory of stability conditions. These concepts provide the framework for our subsequent results of how stability conditions on $D^b (X)$ can be utilized to study the stability of vector bundles and their restrictions on curves.

\subsection{The heart of coherent sheaves}

We recall the slope of a coherent sheaf $E$ shifted by $\beta$: 

\begin{equation} \label{defimu}
    \mu_\beta (E) \colon = \begin{cases} 
    \frac{H.c_1 (E)}{rk(E)} - \beta & \text{if } rk(E)>0 \\
    +\infty & \text{otherwise} 
    \end{cases}
\end{equation} 

\begin{defi}
We say that $E \in Coh(X)$ is $\mu_\beta$-(semi)stable if for all subsheaves $A \subsetneq E$, we have $\mu_\beta (A) < (\leq) \mu_\beta (E/A)$
\end{defi}

The introduction of the slope function $\mu_\beta$ generalizes the classical notion of slope that depends solely on the hyperplane section $H$. By definition, when $\beta=0$, the slope $\mu_\beta$ reduces to the standard slope on the abelian category $Coh(X)$. This enrichment of the category of coherent sheaves $Coh(X)$ with the parameter $\beta \in \mathbb{R}$ allows us to establish a broader set of properties that extend the well-known results obtained using the classical slope. These generalized properties play a crucial role in our analysis of the stability conditions on the K3 surface $X$ and the associated syzygy bundles.

\begin{prop}
\begin{itemize}
    \item Every sheaf $E$ has a (unique and functorial) Harder-Narasimhan filtration (HN-filtration): 
    $$0 = E_0 \subset E_1 \subset E_2 \subset ... \subset E_m=E$$ 
    of coherent sheaves where $E_i / E_{i+1}$ is $\mu_\beta$-semistable for $1 \leq i \leq m$, and with $$\mu^+_\beta (E) \colon = \mu_\beta (E_1/E_0) > \mu_\beta (E_2/E_1) > ... > \mu_\beta (E_m / E_{m-1}) = \colon \mu_\beta^- (E)$$
    \item If $E,F$ are slope-semistable with $\mu_\beta (E) > \mu_\beta (F)$, then $Hom(E,F)=0$
\end{itemize}
\end{prop}

We use the existence of Harder-Narasimhan filtrations for $\mu_\beta$ to construct a specific torsion pair that decomposes the abelian category into two pieces. This torsion pair, denoted by $(T^\beta, F^\beta)$, is crucial for our analysis as it allows us to study the stability of objects within the category. By exploiting the structure of the torsion pair, we can relate the existence of Harder-Narasimhan filtrations to the stability properties of objects in the category, providing a deeper understanding of the stability conditions.

\begin{align*}
    T^\beta & = \{ E \in Coh (X): \mu^-_\beta (E) > 0 \} 
    \\ & = \{ E \in Coh (X) : \text{ all HN-factors of } E \text{ satisfy } \mu_\beta (\cdot) > 0 \} 
    \\ & = \{ E \in Coh (X) : \text{ all quotients } E \rightarrow Q \rightarrow 0 \text{ satisfy } \mu_\beta (Q) > 0\} 
    \\ & = \langle E \in Coh (X) : E \text{ is slope-stable with } \mu_\beta (E)>0 \rangle 
    \\ F^\beta & = \{ E \in Coh(X) : \mu^+_\beta (E) \leq 0 \} 
    \\ & = \{ E \in Coh(X) : \text{ all HN-factors of } E \text{ satisfy } \mu_\beta (\cdot) \leq 0 \} 
    \\ & = \{ E \in Coh(X) : \text{ all subobjects } 0 \rightarrow A \rightarrow E \text{ satisfy } \mu_\beta (A) \leq 0  \} 
    \\ & = \langle E \in Coh(X) : E \text{ is slope-stable with } \mu_\beta (E) \leq 0 \rangle 
\end{align*}
Here, the notation $\langle \cdot \rangle$ denotes the smallest subcategory of $Coh(X)$ that includes the given objects and is closed under extensions.

\begin{remark}
The pair $(T^\beta , F^\beta)$ is a torsion pair, i.e: 
\begin{itemize}
    \item For $T \in T^\beta , F \in F^\beta$, we have $Hom(T,F)=0$
    \item Each $E \in Coh(X)$ fits into a (unique and functorial) short exact sequence $$0 \rightarrow T(E) \rightarrow E \rightarrow F(E) \rightarrow 0$$
    with $T(E) \in T^\beta , F(E) \in F^\beta$
    \end{itemize}
\end{remark}

With the aid of a torsion pair, we can use homological algebra tools, such as tilting, to construct a new abelian subcategory $\Aab$ of the bounded derived category $D^b (X)$. This abelian subcategory $\Aab$, known as the heart of the torsion pair, captures the essential information of the objects in $D^b (X)$ or their translations. An important property of this heart $\Aab$ is that its associated Grothendieck group coincides with the Grothendieck group of the entire derived category $D^b (X)$.

\begin{prop}[\text{\cite[Corollary 2.2]{happel1996tilting}}] \label{Cohbeta}
The following (equivalent) characterization define an abelian subcategory of $D^b (X)$:
\begin{align*}
    Coh^\beta (X)  & = \langle T^\beta , F^\beta [1 ] \rangle \\ & = \{ E \in D^b (X) : H^0 (E) \in T^\beta , H^{-1} (E) \in F^\beta , H^i (E) = 0 \text{ for } i \neq 0,1 \} \\ & = \{ E \in D^b (X) : E \cong (F_{-1} \xrightarrow{d} F_0 ) , ker(d) \in F^\beta , coker(d) \in T^\beta  \}
\end{align*}
\end{prop}

Since we aim to leverage the Grothendieck group of the heart $\Aab=Coh^\beta (X)$, which coincides with the Grothendieck group of the bounded derived category $D^b (X)$, we need to study the short exact sequences within the abelian category $\Aab$. These short exact sequences in $\Aab$ correspond precisely to the exact triangles in $D^b (X)$ of the form $$A \xrightarrow{a} E \xrightarrow{b} B \rightarrow A [1]$$ where all the objects $A, E$ and $B$ belong to the heart $\Aab$, using the two-terms structure of $\Aab$ we have that the short exact sequence can be expressed as follows:
\[
\xymatrix{
A: \ar[d]^a & A_{-1} \ar[d]^{a_{-1}} \ar[r]^{d_A} & A_0 \ar[d]^{a_{0}} \\
E: \ar[d]^b & E_{-1} \ar[d]^{b_{-1}} \ar[r]^{d_E} & E_0 \ar[d]^{b_{0}} \\
B: & B_{-1} \ar[r]^{d_B} & B_0 \\
}
\]
with all squares commutative, columns exact with $a_i$ injective and $b_i$ surjective.

Since we have $T^\beta \hookrightarrow Coh^\beta (X)$ as $T \mapsto T_\cdot = (0 \xrightarrow{0} T) \in Coh^\beta (X)$, if $F \in F^\beta$ then 

\xymatrix{
F_\cdot : ... \ar[r]^0 & 0 \ar[r]^0 & \underbrace{ F }_{\text{index } 0} \ar[r]^0 & 0 \ar[r]^0 & ... & \in D^b (X) \\ F_\cdot [1] : ... \ar[r]^0 & \underbrace{ F }_{\text{index } -1} \ar[r]^0 & 0 \ar[r]^0 & 0 \ar[r] & ... & \in D^b (X)
}

and $F\mapsto F_\cdot [1] = (F \xrightarrow{0} 0 ) \in Coh^\beta (X)$ where $F \mapsto F_\cdot$ is the same morphism that before, that is, the injection in the index 0.

In particular, every object $E \in Coh^\beta (X)$ fits into a short exact sequence 
$$H^{-1}(E)_\cdot \hookrightarrow E \twoheadrightarrow H^0 (E)_\cdot $$
The isomorphism class of $E \in Coh^\beta (X)$ is determined by the extension class \[Ext^1_{D^b (X)} (H^0 (E)_\cdot , H^{-1} (E)_\cdot [1] )\] that is equivalent to $Ext^2_{Coh(X)} (H^0 (E), H^{-1} (E))$ as:
\[
\xymatrix{
H^{-1} (E)_\cdot [1] : \ar[d] & H^{-1} (E) \ar[d] \ar[r]^0 & 0 \ar[d]^0\\
E : \ar[d] & E_{-1} \ar[d]^0 \ar[r]^d & E_0 \ar[d]\\
H^0 (E)_\cdot : & 0 \ar[r]^0 & H^0 (E)
}
\]

Now, since we can see the objects $E$ of $T^\beta$ as objects in $Coh^\beta (X)$, denoted as $E_\cdot$, we want to study a characterization of the subobjects $A_\cdot \hookrightarrow E_\cdot$ in $Coh^\beta (X)$ using properties of the abelian category $Coh(X)$.

\begin{prop}[\text{\cite[Proposition 2.4]{bayer2015wall}}] \label{propsubobj}
Let $E \in T^\beta$ and $E_\cdot$ considered as an object of $Coh^\beta (X)$. To give a subobject $A_\cdot \hookrightarrow E_\cdot$ of $E$ respect to the category $Coh^\beta (X)$ is equivalent of giving a sheaf $A \in T^\beta$ with a map $f:A \rightarrow E$ whose kernel (as a map in $Coh(X)$) satisfies $ker(f) \in F^\beta$.
\end{prop}

In order to reply the additive property of $deg(\cdot)$ and $rank(\cdot)$ under the short exact sequences, we define the Grothendieck group, denoted as $K(\Aab)$, for an abelian category $\Aab$. This group is constructed as the quotient of the free abelian group generated by the objects of $\Aab$, under the relation $[B] = [B']+[B'']$ for any short exact sequence of the form $ 0 \rightarrow B' \rightarrow B \rightarrow B'' \rightarrow 0$ in $\Aab$. For example, if \[0 \rightarrow \mathcal{O}_X(-1) \rightarrow \mathcal{O}_X \rightarrow \mathcal{O}_C \rightarrow 0\] is a short exact sequence in $Coh(X)$, where $C$ is a curve on the surface $X$, then $[\mathcal{O}_X] = [\mathcal{O}_X(-1)] + [\mathcal{O}_C]$ in $K(X)$. This relation reflects the additivity of the degree and rank under short exact sequences.
 
In our case, for the abelian category $\Aab=Coh(X)$ the Grothendieck group denoted as $K(X) := K(Coh(X))$ is generated by the classes of vector bundles $[F]$ on the variety $X$, modulo the relation defined above. This group, $K(X)$, provides a convenient way to encode numerical invariants of coherent sheaves, such as the rank and degree, in a linear algebraic setting.

We fix a finite rank lattice $\Lambda$ (that is, a free abelian group with finite rank) and a surjective group homomorphism $\nu :K(\Aab) \rightarrow \Lambda.$ This allows us to define numerical invariants of objects in $Coh(X)$ by considering their images under $\nu$. For instance, the rank and degree of a vector bundle $F$ can be recovered as $\nu([F]) = (rank(F), deg(F))$.

\subsection{Stability conditions on abelian categories}

In this subsection, we explore the fundamental concepts related to stability conditions on abelian categories. The formal definition of stability conditions was introduced by Bridgeland in \cite{bridgeland2007stability}.

Bridgeland's framework provides a powerful tool for studying the structure of the bounded derived category $D^b (X)$ of a variety $X$. Central to this approach is the notion of a stability condition, which endows the objects in $D^b (X)$ with a notion of (semi)stability. These stability conditions are parametrized by a manifold, known as the stability manifold, which exhibits a rich wall and chamber structure.

For the specific case of K3 surfaces, we focus on a 2-dimensional family of stability conditions. This specialized setting allows us to gain a more refined understanding of the structure of the walls and chambers in the stability manifold.

The key ideas and properties of Bridgeland stability conditions on abelian categories will be introduced in the following subsection, laying the groundwork for our subsequent results.

\begin{defi}
A weak stability function on an abelian category $\Aab$ is a group homomorphism $Z : \Lambda \rightarrow \mathbb{C}$ such that for any $E \in A$, $$Z(\nu (E)) = m (\nu (E)) exp(i \pi \phi( \nu (E)))$$ where $m(\nu(E)) \geq 0$ and $0<\phi (\nu (E)) \leq 1$
\end{defi}

If for any non-trivial object $E$, we have $Z(\nu (E)) \neq 0$, the homomorphism $Z$ is called a \textit{stability function}. If $Z(\nu (E)) = 0$ for a non-trivial object $E \in A$, then we define $\phi (\nu (E)) = 1$. The real number $\phi (\nu (E)) \in (0,1]$ is called the phase of the object $E$. For abuse notations we write $Z(E)$ and $\phi (E)$ instead of $Z(\nu (E))$ and $\phi (\nu (E))$.

\begin{defi}
A non-zero object $E \in \Aab$ is said to be Z-(semi)stable when Z is a stability function if $$0 \neq E' \subsetneq E \Longrightarrow \phi (E') < \phi (E) (\leq \text{ resp.})$$
\end{defi}

We say that the stability function $Z$ satisfies the Harder-Narasimhan property if every non-zero object $E \in \Aab$ has a finite filtration $$0=E_0 \subsetneq E_1 \subsetneq ... \subsetneq E_{n-1} \subsetneq E_n=E$$

whose factors $F_i = E_i/E_{i+1}$ are $Z$-semistable and $$\phi^+ (E) = \phi (F_1) > \phi (F_2) > ... > \phi (F_n) = \phi^- (E)$$

\begin{defi}
Pick a norm $|| \cdot ||$ on $\Lambda_\mathbb{R} = \Lambda \otimes \mathbb{R}$. A (weak) stability function $Z$ on an abelian category $\Aab$ satisfies the support property if there exists a constant $C>0$ such that for all $Z$-semistable objects $0 \neq E$, we have $$||\nu (E)|| \leq C |Z(\nu (E))|$$
\end{defi}

The support property plays a crucial role in endowing the set of stability functions with a geometric structure. By comparing stability functions to norms in an Euclidean space, the support property allows us to define a metric and topology on the space of stability functions on an abelian category $\Aab$. 

Building upon this foundation, we can extend the notion of a stability function from abelian categories to the bounded derived category $D^b (X)$. The key idea is to leverage the concept of the heart $\Aab$ of a t-structure. Since the heart $\Aab$ and the entire derived category $D^b (X)$ share the same Grothendieck group, we can define a stability function on $D^b (X)$ that restricts to a stability function on $\Aab$. This extension allows us to study the stability of complexes in $D^b (X)$ using the same framework as for objects in the abelian category $\Aab$.

\begin{defi}
A (weak) stability condition on the bounded derived category $D^b (X) = D^b (Coh(X))$ is a pair $v=(Z,\Aab)$ where $\Aab$ is the heart of a bounded t-structure on $D^b (X)$ and $Z$ is a (weak) stability function on the abelian category $\Aab$ which satisfies the Harder-Narasimhan property and the support property.
\end{defi}

If $v=(Z,\Aab)$ is a stability condition on $D^b (X)$ an object $E \in D^b (X)$ is said to be $v$-(semi)stable if a shift $E[k]$ is contained in the abelian category $\Aab$ and the object $E[k]$ is (semi)-stable with respect to the stability function $Z$.

The definition of stability conditions on a triangulated category, such as the bounded derived category $D^b(X)$, relies crucially on the concept of the heart of a t-structure. This abelian subcategory $\Aab$ plays a central role in Bridgeland's framework, as it provides the appropriate setting for defining a stability function and the associated notion of (semi)stability.

The construction of the heart $\Aab$ is not a trivial task and often involves the use of homological algebra tools, such as tilting. By performing a tilting procedure, we can produce new hearts of t-structures that exhibit different properties and allow for a more refined analysis of the stability conditions.

Furthermore, the concept of a slicing is intimately connected to the heart of a t-structure. A slicing is a parametrization of the subcategories of the triangulated category using the real numbers. This parametrization enables us to consider families of objects that belong to specific subcategories within given intervals. These families of objects are crucial for the definition and study of stability conditions on $D^b (X)$.

\begin{defi} \label{slicing}
A slicing $P$ of a triangulated category $D$ consist of full additive subcategories $P(\phi)$ for each $\phi \in \mathbb{R}$ satisfying the following axioms:
\begin{enumerate}
    \item $\forall \phi \in \mathbb{R}$, $P (\phi +1) = P(\phi)[1]$
    \item If $\phi_1 > \phi_2$ and $A_i \in P(\phi_i)$ then $Hom_D (A_1,A_2)=0$
    \item for each non-zero object $E \in D$ there is a finite sequence of real numbers $$\phi^+ (E) = \phi_1 > \phi_2 > ... > \phi_n = \phi^- (E)$$ and a collection of triangles
    \[
    \xymatrix{
    0=E_0 \ar[rr]  & & E_1 \ar[rr] \ar[ld]  & & E_2 \ar[ld] \ar[r] \ & ...  \ar[r] & E_{n-1} \ar[rr]  & & E_n \ar[ld] \\ 
     & F_1 \ar[lu] &  & F_2 \ar[lu] &   &   &  &   F_n \ar[lu] & 
    }
    \]
    with $F_i \in P(\phi_i)$ for all $i$
\end{enumerate}
\end{defi}

Any weak stability condition $v=(Z,\Aab)$ defines a slicing $P_v$ (which depends of $v$) of $D^b (X)$ as follows: for each $\phi \in (0,1]$, let $P_v (\phi)$ be the full additive subcategory of $D^b (X)$ of semistable objects with phase $\phi$, together with $0$. The part 1 of the Definition \ref{slicing} determines $P_v (\phi)$ for all $\phi \in \mathbb{R}$. Then, to refer to a stability condition $\nu$, we can use the pair $(Z,\Aab)$ where $\Aab$ is the hearth of a bounded t-structure on $D^b (X)$ or the pair $(Z,P_\nu)$ where $P_\nu$ is a slicing of $D^b (X)$.

\begin{teo}[\text{\cite[Lemma 6.2]{bridgeland2008stability}}] \label{teoalfbet}
For each $\beta,\alpha \in \mathbb{R}$ with $\alpha>0$, considerer the pair $\sigma_{\beta,\alpha} = (Z_{\beta,\alpha}, Coh^\beta (X))$ with $Coh^\beta (X)$ as below and with $Z_{\beta,\alpha} : K (D^b (X)) \rightarrow \mathbb{C}$ defined by: \begin{align*}
    Z_{\beta,\alpha} (E) & = \frac{\alpha^2 H^2 - b^2 H^2}{2} v_0 (E) + \beta H v_1 (E) - v_2 (E) + i \alpha H (v_1 (E) - \beta H rk(E)) \\ & = \langle exp(i \alpha H + \beta H) , v(E) \rangle \\ & = \langle exp(i \alpha H + \beta H) , (ch_0 (E) , ch_1 (E) , ch_2 (E) + ch_0 (E) \rangle 
\end{align*}
This pair defines a Bridgeland stability condition on $D^b (X)$ if $\mathfrak{Re} (Z_{\beta,\alpha} (\delta) ) > 0$ for all roots $\delta \in H^\ast_{alg} (X;\mathbb{Z})$ of the form $(r,r\beta,s)$ with $r>0$ and $s \in \mathbb{Z}$ arbitrary; in particular, this holds for $\alpha^2 H^2 \geq 2$.

Moreover, the family of stability conditions $\sigma_{\beta,\alpha}$ varies continuously as $(\beta,\alpha)$ vary in $\mathbb{R} \times \mathbb{R}_{>0}$.
\end{teo}

In order to explain the notation we introduce the Mukai vector of an object $E$. 
\vspace{1mm}

The Mukai vector of an object $E \in D^b (X)$ given by $$v(E) = (v_0 (E), v_1 (E) , v_2 (E)) = (ch_0 (E),ch_1 (E) , ch_2 (E) + ch_0 (E))$$
lies in the algebraic cohomology $H^\ast_{alg} (X;\mathbb{Z})$. The pairing $\langle \text{ } , \text{ } \rangle $ is the Mukai pairing $$\langle v(E),v(F) \rangle = - \chi (E,F) = v_1(E) v_1 (F) - v_0 (E) v_2 (F) - v_2 (E) v_0 (F)$$

For each sheaf $E$, we have $\mathfrak{Im} (Z_{\beta,\alpha} (E))\geq 0$ if and only if $\mu_\beta (E) \geq 0$. And $Z_{\beta,\alpha} (F[1])=-Z_{\beta,\alpha}(F)$.

\vspace{1mm}

Let $Stab(X)$ be the set of stability conditions of $D^b (X)$ (with respect to the lattice $\Lambda$ and the vector $v$). This set can be enriched with a topology as the coarsest topology such that for any $E \in D^b (X)$ the maps $(Z,\Aab) \mapsto Z$, $(Z,\Aab) \mapsto \phi^+ (E)$ and $(Z,\Aab) \mapsto \phi^- (E)$ are continuous.

\begin{teo}[\text{\cite[Theorem 1.2]{bridgeland2007stability}}]
The map $\mathcal{Z} : Stab(X) \rightarrow Hom (\Lambda, \mathbb{C})$ given by $(\Aab,Z) \mapsto Z$ is a local homeomorphism. In particular, $Stab(X)$ is a complex manifold of dimension $rk(\Lambda)$.
\end{teo}

We can study the behavior of an object $E \in Coh(X)$ when we let vary the stability condition in $Stab(X)$. We want to study the sets of stability conditions for which $E$ is stable, semistable or unstable.

\begin{defi} \label{defiwall}
Let $v_0,w \in \Lambda-{0}$ be two non-parallel vectors. A numerical wall $W_w (v_0)$ for $v_0$ with respect to $w$ is a non-empty subset of $Stab(X)$ given by $$W_w (v_0) = \{ \sigma = (Z,P) \in Stab(X) : \mathfrak{Re} (Z(v_0)) \cdot \mathfrak{Im} (Z(w)) = \mathfrak{Re} (Z(w)) \cdot \mathfrak{Im} (Z(v_0))\}$$
\end{defi}

The wall and chamber structure of the stability manifold $Stab(X)$ is intimately related to the topology of this space. We denote by $\mathcal{W} (v_0)$ the set of numerical walls for a fixed Mukai vector $v_0$. These numerical walls are real codimension-1 submanifolds of $Stab(X)$, which divide the stability manifold into connected components called chambers.

\begin{sloppypar}
This wall and chamber structure encodes crucial information about the behavior of (semi)stability for objects with fixed Mukai vector $v_0$. As we vary the stability condition by moving within a chamber, the (semi)stability of objects with Mukai vector $v_0$ remains unchanged. However, when crossing a wall, the (semi)stability can change abruptly. By understanding this wall and chamber structure, we gain a global perspective on how the stability of objects depends on the choice of stability condition.
\end{sloppypar}

Furthermore, the topology of $Stab(X)$ is closely related to the wall-and-chamber decomposition. The connected components of the complement of the union of all walls are precisely the chambers. Since the walls form a locally finite arrangement, the topology of $Stab(X)$ remains well behaved.

In order to study the case of K3 surfaces, we want to describe the wall and chamber structure on the $(\beta,\alpha)$-plane for the $\sigma_{\beta,\alpha}$ stability conditions as in \cite{macri2017lectures} section 6.4.

Consider $T=\{ v \in K (X) : \chi (v,w)=0 \text{ for all } w \in K (X) \}$ and  the numerical Grothendieck group $K_{num} (X) = K (X) / T$ as a finitely generated $\mathbb{Z}$-lattice. In the case where $X$ is a K3 surface, the group $K_{num} (X)$ coincides with the algebraic cohomology group $H^\ast_{alg} (X)$. 

\begin{prop}[\text{\cite[Proposition 6.22]{macri2017lectures}}]
Fix a class $v \in K_{num} (X)$.
\begin{enumerate}
    \item All numerical walls are either semicircles with center on the $\beta$-axis or vertical rays.
    \item Two different numerical walls for $v$ cannot intersect.
    \item For a given class $K_{num} (X)$ the hyperbola $\mathfrak{Re} (Z_{\beta,\alpha} (V))=0$ intersects all numerical semicircular walls at their top points.
    \item If $ch_0 (v) \neq 0$, then there is an unique numerical vertical wall defined by the equation $$\beta = \frac{H \cdot ch_1 (v)}{H^2 \cdot ch_0 (v)}$$
    \item If $ch_0 (v) \neq 0$, then all semicircular walls to either side of the unique numerical vertical wall are strictly nested semicircles.
    \item If $ch_0 (v) = 0$, then there are only semicircular walls that are strictly nested.
    \item If a wall is an actual wall at a single point, it is an actual wall everywhere along the numerical wall.
\end{enumerate}
\end{prop}

Building upon the foundational result established in Theorem \ref{teoalfbet}, which states that the pair $\sigma_{\beta, \alpha} = (Z_{\beta,\alpha}, Coh^\beta (X))$ defines a stability condition on the bounded derived category $D^b (X)$, we now turn our attention to understanding the behavior of the stability of coherent sheaves as we navigate the wall and chamber structure of the $(\beta,\alpha)$-plane.

\begin{cor}[\text{ \cite[Corollary 3.5]{bayer2015wall} }]
Given a class $v \in H^\ast_{alg} (X;\mathbb{Z})$. For objects of Mukai vector $v$, being $\sigma_{\beta,\alpha}$-(semi)stable is independent on the choice of $(\beta,\alpha)$ in any given chamber.
\end{cor}

Now that we have established the independence of (semi)stability on the choice of stability condition within a given chamber of the $(\beta , \alpha)$-plane, the natural next step is to study the moduli spaces that parametrize the semistable objects for a fixed Mukai vector $v$. These moduli spaces encode crucial information about the behavior of (semi)stable objects as we vary the stability condition by moving between different chambers in the $(\beta,\alpha)$-plane.

\subsection{Moduli spaces of stable objects}

In this subsection, we delve into the study of moduli spaces that parametrize stable objects with respect to a given stability condition. Our goal is to understand how the structure of these moduli spaces varies as we change the stability condition. Furthermore, we investigate whether there exist stability conditions for which the associated moduli spaces coincide with the moduli spaces obtained under classical stability notions, such as Gieseker stability.

To lay the groundwork for this analysis, let us first recall some fundamental results regarding the structure and non-emptiness of moduli spaces of stable objects. These results serve as a foundation for our subsequent investigations of how the moduli spaces depend on the choice of stability condition.

A crucial question is whether these moduli spaces are non-empty for a given stability condition. This is closely related to the existence of stable objects for that condition. By leveraging the wall-crossing techniques developed in the previous sections, we establish conditions under which the moduli spaces are non-empty and study how their structure changes as we vary the stability condition.

\begin{teo}[\text{ \cite[Theorem 4.1]{bayer2015wall} }] \label{modulialfbet}
Consider a vector $v \in H^\ast_{alg} (X; \mathbb{Z})$, and let $\sigma = \sigma_{\beta,\alpha}$ be a stability condition that is not on any of the walls for the wall and chamber decomposition with respect to $v$. Then the coarse moduli space $M_\sigma (v)$ of $\sigma_{\beta,\alpha}$-stable objects of Mukai vector $v$ exists as a smooth projective irreducible holomorphic symplectic variety. It is non-empty if and only if $v^2 \geq -2$ and its dimension is given by $dim M_\sigma (v) = v^2 +2$
\end{teo}

Consider now the moduli space of $H$-Giesker-stable sheaves and lets compare this set with the moduli space of $\sigma_{\beta,\alpha}$-stable objects with $\alpha >> 0$.

\begin{teo}[\text{ \cite[Theorem 4.4]{bayer2015wall} }] \label{teomodul}
Let $v=(v_0,v_1,v_2)$ a class in $H^\ast_{alg} (X ; \mathbb{Z})$ having either positive rank $v_0>0$, or satisfying $v_0=0$ with $v_1$ being effective. Then there exists $\alpha_0$ such that for all $\alpha \geq \alpha_0$ and all $\beta > \frac{H \cdot v_1}{H^2 \cdot v_0}$ (or $\beta > \frac{v_2}{H \cdot v_1}$ in case $v_0=0$), the moduli space $M_{\sigma_{\beta , \alpha}} (v)$ is equal to the moduli space $M_H (v)$ of $H$-Gieseker-stable sheaves of class $v$. More precisely, and object $E \in D^b (X)$ with $v(E)=v$ is $\sigma_{\beta, \alpha}$-stable if and only if it is the shift of a Gieseker-stable sheaf.
\end{teo}

\section{Stability on K3 surfaces}

In this chapter, we consider a pair $(X,H)$ where $X$ is a smooth K3 surface over $\mathbb{C}$ and $H$ is an hyperplane section of $X$ such that
\begin{itemize}
    \item[$(\ast \ast)$] 
    \begin{center}
     $H^2$ divides $H.D$ for all curve classes $D$ on $X$
    \end{center} 
\end{itemize} 
An example of such pairs is the case when $H$ is an hyperplane section of $X$ such that $Pic(X) \cong \mathbb{Z} . H$. For a more detailed discussion of the stability conditions in this chapter, please see Appendix B. Throughout this chapter, we will reference the necessary theory to clarify the notation.

Let $d,g \in \mathbb{Z}$ with $g>1$. The moduli space of H-Gieseker stable sheaves with Mukai vector $v=(0,H,d+1-g)$ is $M_H (v)$ (see \ref{teoalfbet}), the moduli space $M_H (v)$ parametrizes one-dimensional sheaves $F$ of Euler characteristic $d+1-g$ where the support $|F|$ corresponds to a curve in the linear system $|H|$. 

\vspace{1mm}

We consider the stability condition $\sigma_{\beta,\alpha}$ associated to $\alpha,\beta \in \mathbb{R}$ with $\alpha>0$. We study the wall-crossing for the moduli space $M_{\sigma_{\beta, \alpha}} (v)$ (see \ref{modulialfbet}), with $v=(0,H,d+1-g)$ as above. From Theorem \ref{teomodul} we have $M_{\sigma_{\beta, \alpha}} (v) = M_H (v)$ for $\alpha >> 0$, and we want to find the walls that bounds this chamber called the Gieseker-chamber, for wall-chamber structure see \ref{defiwall}.

Consider $\beta=0$. In this case $\mathfrak{Im}(Z_{0, \alpha} (\OO_X))=0$, this means we have stability conditions for $$\alpha > \alpha_0 = \sqrt{\frac{2}{H^2}},$$ for details see \ref{teoalfbet}.
For these stability conditions, notice that $\OO_X [1]$ is an object in the category $Coh^0 (X)$ with \[ \mathfrak{Im} (Z_{0,\alpha} (\OO_X [1]))=0, \] i.e. of slope $\mu_0 (\OO_X [1])=+\infty$ (see \ref{defimu}), therefore it is automatically semistable. Using Proposition \ref{propsubobj} $\OO_X [1]$ has no subobjects in $Coh^0 (X)$ (see \ref{Cohbeta}), and so $\OO_X [1]$ is stable for $\beta=0$.

\begin{lema}[\text{\cite[Lemma 6.1]{bayer2015wall}}] \label{Lemmm421}
For $\alpha> \alpha_0$ and $\beta=0$, we have an isomorphism $M_{\sigma_{0, \alpha}} (v) = M_H (v)$ identifying the stable objects with stable sheaves.
\end{lema}

In other words, there is no wall intersecting the line segment $\beta=0, \alpha= \left( \frac{2}{H^2} , \infty \right).$
The interplay between Brill-Noether theory and wall-crossing theory is fundamental in this section. While Brill-Noether theory allows us to study properties of the Lazarsfeld-Mukai sheaves $F_{V,L}$ associated to a linear series $(L,V)$, wall-crossing theory provides the tools to analyze the stability of these sheaves when restricted to curves living on K3 surfaces. The conditions presented in Lemma \ref{Lem3} complement the result presented in Lemma \ref{Lemmm421}.

\begin{lema}[\text{\cite[Lemma 6.2]{bayer2015wall}}] \label{Lem3}
There is a wall bounding the Gieseker-chamber where $Z_{\beta,\alpha} (\OO_X)$ aligns with $Z_{\beta,\alpha} (v)$. The sheaves $L \in M_{\sigma_{\beta,\alpha}} (v)$ getting destabilised are exactly those with $h^0 (L)>0$, and the destabilising short exact sequence are given by \begin{equation} \label{eq:1}
    H^0(L)\otimes \OO_X \hookrightarrow L \twoheadrightarrow W
\end{equation} for some object $W$ that remains stable at the wall.
\end{lema}

\begin{lema}
Let $\Bar{\sigma}$ be a stability condition on the wall constructed above. Let $W$ be an object of class $v - t v(\OO_X)$ for some $t \in \mathbb{Z}$, and assume that $W$ is $\Bar{\sigma}$-semistable. Then, $W$ is stable if and only if $Hom(\OO_X,W)=0=Hom(W,\OO_X)$. \label{lemabridg}
\end{lema}

When $L$ is globally generated, the object $W$ in \ref{eq:1} is the shift $F_{L} [1]$ of the kernel bundle $F_L$ of the evaluation map $H^0 (L) \otimes \OO_X \xrightarrow{ev} L$, called the Lazarsfeld-Mukai bundle. This result establishes the stability of the Lazarsfeld-Mukai bundle $F_{L}$ associated to a complete linear series $(L,H^0 (L))$ on a curve $C \in |H|$.
For a non-complete generated linear series $(L,V)$ on $C \in |H|$, we can construct the Lazarsfeld-Mukai bundle associated to the evaluation map $V \otimes \OO_X \xrightarrow{ev} L$, and we have the  the following,

\begin{cor}
Suppose $\beta<0$. If $L$ is generated by $V \in Gr(r+1,H^0 (L))$ then \ref{eq:1} have the form $$V\otimes \OO_X \hookrightarrow L \twoheadrightarrow F_{V,L} [1]$$
with $F_{V,L}[1]$ strictly-semi-stable on the wall and stable on the side of the wall where $L$ is stable.
\end{cor}

\begin{proof}
First, notice that $\OO_X \otimes H^0(L) \hookrightarrow L \twoheadrightarrow W$. If $L$ is globally generated then the evaluation map $ev$ is surjective as a map in $Coh(X)$ leading to the inclusion $\OO_X \otimes H^0(L) \hookrightarrow L \twoheadrightarrow W$, where $W = F_L [1]$ and $W_1 = F_{V,L} [1]$. Let $\iota : V \rightarrow H^0 (L)$ be the inclusion, from Lemma \ref{lemabridg} $Hom(W,\OO_X) = 0$ since $W_1$ and $W$ are quotient of $L$, which implies $Hom(W,\OO_X) = Hom(W_1 , \OO_X) = 0$ as well. Consider the diagram
\[
\xymatrix{
V \otimes \OO_X \ar[r] & L \ar[r] & W_1 \\
H^0(L) \otimes \OO_X \ar[r] \ar[u]^{\iota \otimes Id} & L \ar[u]^{Id} \ar[r] & W \ar[u]
}
\]
By applying $Hom(\OO_X , \cdot)$ we get
\[
\xymatrix{
Hom(\OO_X, V \otimes \OO_X) \ar[d]^{Hom(\OO_X, \iota)} & Hom( \OO_X , L) \ar[l] \ar[d]^{Hom(\OO_X,Id)} & Hom(\OO_X , W_1) \ar[d] \ar[l] \\
Hom(\OO_X,H^0(L) \otimes \OO_X) & Hom(\OO_X,L) \ar[l] & Hom(\OO_X,W) \ar[l]
}
\]
Let us denote the dimension of $dim (Hom(\cdot))$ by $hom(\cdot)$. By Theorem \ref{Lem3}, we have $Hom(\OO_X, W) = 0$ and $hom(\OO_X , H^0 (L) \otimes \OO_X) = h^0(L) = hom(\OO_X , L)$. Therefore, \[
hom(O_X,W_1) = h^0(L)-dim(V) = codim_{H^0(L)} (V).
\]
By Lemma \ref{lemabridg}, it follows that $W_1$ is strictly semistable and it is stable at the same side where $L$ is stable.
\end{proof}

Now that we have established the basic results for studying moduli spaces of stable objects with respect to a given stability condition, we aim to take advantage of these tools to relate the stability of vector bundles to the stability of their restrictions to divisors. By employing wall-crossing techniques, we are able to construct precise connections between the (semi)stability of a vector bundle on a surface $X$ and the (semi)stability of its restriction to a curve $C$ lying on $X$.
This approach  allows us to relate the stability of linear series on curves to the stability of the associated syzygy bundles, which is the main objective of this thesis.

\subsection[Cohomological conditions for the stability of MVL and restriction theorem]{Cohomological conditions for the stability of \texorpdfstring{$\mvl$}{MVL} and restriction theorem}

In this section, we explore the relation among the stability of vector bundles on a K3 surface $X$ and the stability of their restrictions to curves $C$ inside $X$. Specifically, we focus on Lazarsfeld-Mukai bundles $F_{V,L}$ associated to a generated linear series $(L,V)$ on $C$. By leveraging the results on stability conditions for sheaves on K3 surfaces developed in the previous sections, we establish conditions under which the restriction of $F_{V,L}$ to $C$ remains stable.

\vspace{2mm}
We recall results of stability for restrictions from the literature, such as those appearing in \cite{feyzbakhsh2022effective}. These results provide sufficient conditions to ensure that the restriction of a stable bundle remains stable when restricted to a curve. By adapting and extending these conditions to the specific setting of Lazarsfeld-Mukai bundles, we are able to relate their stability to the stability of the associated syzygy bundles $\mvl$ on the curve $C$. The insights gained from this analysis are crucial to state the main results in this thesis, which aim to connect the stability of linear series to the stability of syzygy bundles. By understanding the interplay between the geometry of K3 surfaces and the stability of restrictions, we aim to use the stability of Lazarsfeld-Mukai bundles to draw conclusions about the stability of syzygy bundles.

\vspace{2mm}
Let $X$ be a smooth complex projective variety of dimension $n \geq 2$ with an ample divisor $H$. For a $\mu$-stable coherent sheaf $E$ of positive rank on $X$, we have $$\Tilde{\Delta} (E) = \left( \frac{ch_1(E).H^{n-1}}{ch_0(E) H^n} \right)^2 - 2 \frac{ch_2 (E).H^{n-2}}{ch_0 (E) H^n}$$
\begin{itemize}
    \item $\mu^{max} (E) = max \{ \mu(F) : F \text{ is a subsheaf of } E \text{ with } \mu(F) < \mu(E) \}$
    \item $\mu^{min} (E) = min \{ \mu(F') : F' \text{ is a proper quoatient sheaf of } E \}$
\end{itemize}
and $\delta (E) = min\{ \mu^{min}(E) - \mu (E) , \mu(E) - \mu^{max} (E) \}.$

\begin{teo}[\text{\cite[Theorem 1.1]{feyzbakhsh2022effective}}]
    Let $E$ be a $\mu$-stable reflexive sheaf on $X$ of rank $r>0$. The restricted sheaf $E|_D$ for any irreducible divisor $D \in |mH|$ is $\mu$-semistable on $D$ if $$ m \geq \frac{r+2}{\sqrt{r+1}} \sqrt{\Tilde{\Delta} (E)} \text{ and } \frac{m}{2} + \sqrt{\frac{m^2}{4}- \Tilde{\Delta} (E)} \geq \frac{\Tilde{\Delta} (E)}{\delta(E)}.$$ Moreover, $E|_{D}$ is $\mu$-stable if the inequalities are both strict.
\end{teo}

When $r>1$, then $\delta(E) \geq \frac{1}{H^n r(r-1)}$ and we can restrict the conditions as:

\begin{prop}[\text{\cite[Proposition 4.6]{feyzbakhsh2022effective}}]
    Let $E$ be a $\mu$-stable reflexive sheaf as above with rank $r>1$. The restricted sheaf $E|_D$ for any irreducible divisor $D \in |mH|$ is $\mu$-(semi)stable on $D$ if 
    \begin{equation} \label{restriccond}
    m> (\geq) r(r-1) \Tilde{\Delta} (E) + \frac{1}{r(r-1)}.
    \end{equation}
\end{prop}

\vspace{1mm}

For a smooth $K3$ surface $X$, a smooth curve $C\subset X$ and a generated linear series $(L,V)\in G^r_d(C)$, the Lazarsfeld-Mukai bundle is defined via the following elementary modification on $X$
$$0\to F_{C,V,L}\to V\otimes\mathcal O_S\to L\to 0$$
\vspace{1mm}
For short we write $F_{V,L}$ when the context is understood.
\vspace{1mm}
Now, let $C \in |H|$ be a smooth curve with a $(L,V)\in G^r_d(C)$ a generated linear series on $C$ such that $d\leq g-1$, and consider the corresponding Lazarsfeld-Mukai bundle $F_{V,L}$ on $X$. Let's compute the condition \ref{restriccond} for this case: First observe that $ch(F_{V,L})=(r+1,H,g-1-d)$, $n=2$ and 
\begin{align*}
\Tilde{\Delta}(F_{V,L}) = \left(\frac{H.H^1}{r.H^2} \right)^2 - 2\frac{(g-1-d).H^0}{(r+1)H^2} &= \frac{1}{(r+1)^2}-2\frac{g-1-d}{(r+1)(2g-2)} \\  &= \frac{1}{r+1}\left( \frac{d}{g-1}-\frac{r}{r+1} \right) ,\end{align*}
by substituting the right side of \ref{restriccond} for $m=1$, 
\begin{align*}
    r(r+1) \cdot \frac{1}{r+1} \left( \frac{d}{g-1} - \frac{r}{r+1} \right) + \frac{r}{r+1} &= \frac{rd}{g-1} - \frac{r^2}{r+1} + \frac{1}{r(r+1)} \\ &= \frac{rd}{g-1} + \frac{1-r^3}{r(r+1)}
\end{align*}

Since $d \leq g-1$ then $\frac{d}{g-1}r \leq r$, then for $r\geq 2$,
\begin{align*}
    \frac{rd}{g-1} + \frac{1-r^3}{r(r+1)} &\leq r+\frac{1-r^3}{r(r+1)} \\ 
    & = \frac{r^3+r^2+1-r^3}{r(r+1)} \\
    & = \frac{r^2+1}{r^2+r} \\
    & < 1.
\end{align*}
We conclude that the restriction of the bundle $F_{V,L}$ to the curve $C$, denoted by $F_{V,L}|_C$, is stable. 

\begin{corollary}
    In the above context $F_{V,L}|_C$ is stable.
\end{corollary}

The following short exact sequence relates $F_{V,L}|_C$ with $M_{V,L}$ (see \cite{aprodu2010koszul}). We denote $K^{-1}_C L$ for the line bundle $L \otimes K^{-1}_C$:
\begin{equation*}
    0 \rightarrow K^{-1}_C L \rightarrow F_{V,L}|_C \rightarrow \mvl \rightarrow 0
\end{equation*}

\begin{prop}
    We have the following correspondence between the subbundles of $M_{V,L}$ and the subbundles of $F_{V,L}|_C$:
\[ 
 \left\{ S \subset \mvl \right\} \Longleftrightarrow  \left\{ S' \subset F_{V,L}|_C : K^{-1}_C L \subset S' \right\} 
\]
\end{prop}

\begin{proof}
    \begin{itemize}
        \item[($\Rightarrow)$] Taking $S \subset \mvl$, consider the following diagram:
        \[
        \xymatrix{
         & K_{C}^{-1} L \ar[d]^=  &  & S \ar[d]^\iota  & \\
        0 \ar[r] & K_{C}^{-1} L \ar[r] & F_{V,L}|_{C} \ar[r]  & \mvl \ar[r]  & 0
        }
        \]
        Where $\iota$ is the inclusion, we can complete the diagram by using the pullback of the diagram $F_{V,L}|_C \rightarrow \mvl \xleftarrow{\iota} S$ in the abelian category $Coh(C)$ (see \cite[Lemma 7.29]{rotman2009introduction} for the case of modules, the proof is analogue for abelian categories), called $S'$ and we obtain:
        \[
        \xymatrix{
        0 \ar[r] & K_C^{-1} L \ar[d]^=  \ar[r] & S' \ar[d] \ar[r] & S \ar[d]^\iota  \ar[r] & 0 \\
        0 \ar[r] & K_C^{-1} L \ar[r] & F_{V,L}|_{C} \ar[r]  & \mvl \ar[r]  & 0
        }
        \]
        by uniqueness of the pullback, the correspondence $S \mapsto (K_C^{-1} \hookrightarrow S')$ is injective.
        \item[($\Leftarrow$)] Taking $S'\subset F_{V,L}|_{C}$ with $K_{C}^{-1} L \hookrightarrow S'$, consider $S$ as the quotient $S'/K_{C}^{-1}$, since $S' \hookrightarrow F_{V,L}|_{C}$ then $S \hookrightarrow \mvl$ and the following diagram is commutative
        \[
        \xymatrix{
        0 \ar[r] & K_C^{-1} L \ar[d]^= \ar[r] & S' \ar[d] \ar[r] & S \ar[d] \ar[r] & 0 \\
        0 \ar[r] & K_C^{-1} L \ar[r] & F_{V,L}|_C \ar[r] & \mvl \ar[r]  & 0 
        }
        \]
        
    \end{itemize}
\end{proof}

Consider the following diagram under the correspondence above, with the snake's lemma we have that the last row is an isomorphism.
\begin{equation}
    \xymatrix{
    & &  0 \ar[d] & 0 \ar[d] & \\
    0 \ar[r] & K_C^{-1} L \ar[d]^= \ar[r] & S' \ar[d] \ar[r] & S \ar[d] \ar[r] & 0 \\
    0 \ar[r] & K_C^{-1} L \ar[r] & F_{V,L}|_C \ar[r] \ar[d] & \mvl \ar[r] \ar[d] & 0 \\
    & & F_{V,L}|_C / S' \ar[r]^= \ar[d] & \mvl/S \ar[d] & \\
    & & 0 & 0
    }
    \label{eqivres}
\end{equation}

From now on we denote by $Q$ the quotient $\mvl / S$ or equivalently $ F_{V,L}|_C / S'$. Given the relation between the bundles $\mvl$ and $F_{V,L}|_C$, we aim to take advantage of this connection and the fact that the restriction of $F_{V,L}|_C$ is stable to construct conditions under which this stability implies the stability of $\mvl$.
We recall a result that connects the stability of a stable vector bundle $E$ with the global sections of the elements that appear in a short exact sequence involving $E$.

\begin{lema}[\text{\cite[Lemma 1.1]{russo1997conjecture}}] \label{lemaTeix}
    Let $E$ be a slope-stable vector bundle. Assume that we have an exact sequence \[
    0 \rightarrow E' \rightarrow E \rightarrow E'' \rightarrow 0.
    \]  
    Then $h^0 ( (E'')^\vee \otimes E')=0$.
\end{lema}

This result allows us to conclude that in the case when $d \leq g-1$ then $H^0 (Q^\vee \otimes S')=0$ and $H^0 (\mvl^\vee \otimes K_C^{-1} L)=0$, by considering the exact sequences presented in diagram \ref{eqivres} and the stability of $F_{V.L}|_C$ in this case. 
\vspace{1mm}

Lets recall that we consider polarized $K3$ surfaces $(X,H)$ that satisfy the next property:
\begin{itemize}
    \item[$(\ast \ast)$] 
    \begin{center}
     $H^2$ divides $H.D$ for all curve classes $D$ on $X$
    \end{center} 
\end{itemize} 

\begin{teo} \label{equivk3surf}
    Let $(X,H)$ be a polarized K3 surface satisfying property $(\ast \ast)$ and consider $C \in |H|$ a curve of genus $g>2$, let $(L,V)$ be a generated linear series of type $(d,r+1)$ over $C$, with $1<r<d \leq \min \{ g-1 , k r \}$ where $k$ is the gonality of $C$. If $(L,V)$ is linearly stable then $\mvl$ is slope-stable.
\end{teo}

\begin{proof}
    Let $F_{V,L}$ the Lazarsfeld-Mukai bundle associated to $(L,V)$. We recall that with our hyphotesis the restriction to $C$, $F_{V,L}|_C$, is a slope-stable bundle.
    \vspace{1mm}
        Suposse that $\mvl$ is not slope-stable and let $S$ be a slope-stable maximal destabilizing subbundle of $\mvl$. We denote by $s$ and $d_S$ the rank and degree of the bundle $S$, respectively. 

    \vspace{1mm}
    Since $S$ is slope-stable with $\mu (S) \geq \mu (\mvl)$ then $-\frac{d_S}{r_S} \leq \frac{d}{r}$, and $d \leq g-1$. We can compare the slopes $\mu (F_{V,L}|_C)$ and $\mu(S)$, given that $\mu(F_{V,L}|_C) = -\frac{2g-2}{r+1}$,

    \begin{align*}
        \mu (F_{V,L}|_C) - \mu(S) & = -\frac{2g-2}{r+1} - \frac{d_S}{r_S} \\
        & \leq -\frac{2g-2}{r+1} + \frac{d}{r} \\
        & \leq -\frac{2g-2}{r+1} + \frac{g-1}{r} \\
        & = \frac{g-1}{r} - \frac{2}{r+1} (g-1) \\
        & = (g-1) \frac{1}{r(r+1)} (r+1-2r) \\
        & = \frac{g-1}{r(r+1)} (1-r) \\ 
        & < 0.
    \end{align*}

    Thus, we conclude that $\mu (F_{V,L}|_C) < \mu (S)$, which implies that $Hom (S,F_{V,L}|_C)=0$, or equivalently, $H^0 (S^\vee \otimes F_{V,L}|_C)=0$.

    \vspace{2mm}
    Next, using the short exact sequence $0 \rightarrow S' \rightarrow F_{V,L}|_C \rightarrow Q \rightarrow 0$ and that $F_{V,L}|_C$ is slope-stable, from Lemma \ref{lemaTeix} we obtain $H^0 (S' \otimes Q^\vee ) = 0$ and $H^0 (S' \otimes \mvl^\vee) \hookrightarrow H^0 (S' \otimes S^\vee)$, where this space satisfies $H^0 (S' \otimes S^\vee) \hookrightarrow H^0 (F_{V,L}|_C \otimes S^\vee)$.
    \vspace{2mm}
    
    We aim to establish conditions under which the space $H^0 (S' \otimes \mvl^\vee)$ is non-zero.
        By Serre duality, $H^0 (S' \otimes \mvl^\vee) \cong H^1 (\mvl \otimes (S')^\vee \otimes K_C)^\vee$. To compute the dimension of this space we use \cite[Theorem 2.28]{aprodu2010koszul}, for which we aim to compute $H^1 (\mvl \otimes (S')^\vee \otimes K_C \otimes L)$. Additionally, from Serre duality $H^1 (\mvl \otimes (S')^\vee \otimes K_C\otimes L) \cong H^0 (\mvl^\vee \otimes S' \otimes L^{-1})^\vee$.
    \vspace{1cm}

    \begin{claim}
        $H^0 (\mvl^\vee \otimes S' \otimes L^{-1})=0$.
    \end{claim}
    \vspace{1mm}
    
    \begin{claimproof}
        If $H^0 (\mvl^\vee S' \otimes L^{-1}) \neq 0$, then $L \hookrightarrow \mvl^\vee \otimes S'$, or equivalently, $r+1< h^0(L) \leq h^0 (\mvl^\vee \otimes S')$. Now, consider the short exact sequence obtained by twisting the first exact row of diagram \ref{eqivres} by $S^\vee$ 
        \[
        0 \rightarrow K_C^{-1} L \otimes S^\vee \rightarrow S' \otimes S^\vee \rightarrow S \otimes S^\vee \rightarrow 0.
        \]
        Since $S$ is slope-stable and $K_C^{-1} L$ is a line bundle, we have that $K_C^{-1} L \otimes S^\vee$ is slope-stable with degree
        \[
        deg(K_C^{-1} L \otimes S^\vee) = r_S (d+2-2g)-d_S = r_S (d+1-g)+r_S(1-g)-d_S < r_S (1-g)-d_S < 0.
        \]
        Therefore, $h^0 (K_C^{-1} L \otimes S^\vee) = 0$ and $h^0 (S' \otimes S^\vee) \leq h^0 (S \otimes S^\vee)=1$ by stability of $S$.

        \vspace{1mm}
        On the other hand, by dualizing the right exact column of diagram \ref{eqivres} and twisting by $S'$, 
        \[
        0 \rightarrow S' \otimes Q^\vee \rightarrow S' \otimes \mvl^\vee \rightarrow S' \otimes S' \otimes S^\vee \rightarrow 0.
        \]
        Since $F_{V,L}|_C$ is slope-stable, from Lemma \ref{lemaTeix} we get $h^0 (S' \otimes Q^\vee)=0$ and $h^0 (S' \otimes \mvl^\vee) \leq h^0 (S' \otimes S^\vee) \leq 1$, which leads to a contradiction, since we assumed that $h^0 (\mvl^\vee \otimes S') > r+1$. Therefore, we conclude that $h^0 (\mvl^\vee \otimes S' \otimes L^{-1})=0$.
    \end{claimproof}
    \vspace{2mm}

    Since $H^0 (\mvl^\vee \otimes S' \otimes L^{-1})=0$ and $H^1 ( \mvl \otimes (S')^\vee \otimes K_C \otimes L) \cong H^0 (\mvl^\vee \otimes S' \otimes L^{-1})^\vee$, it follows that $H^1 ( \mvl \otimes (S')^\vee \otimes K_C \otimes L) = 0$. From \cite[Theorem 2.28]{aprodu2010koszul}, we conclude that
    \[
    H^0 (S' \otimes \mvl^\vee)^\vee \cong H^1 (\mvl \otimes (S')^\vee \otimes K_C) \cong K_{r-1,2} (C,\mvl \otimes (S')^\vee \otimes K_C, L, V).
    \]
    Set $\mathcal{E}:=\mvl \otimes (S')^\vee \otimes K_C$. The latter space is isomorphic to (see \cite{aprodu2010koszul}):
    \begin{equation}
        \label{coker1}
        \operatorname{coker}\left(
        \bigwedge^r V \otimes H^0(\mathcal{E}\otimes L)
        \longrightarrow
        H^0\left(\mathcal{E}\otimes\bigwedge^{r-1}\mvl\otimes L^2\right)
        \right).
    \end{equation}
    \vspace{2mm}
    
    To compute the dimension of this vector space, we use the fact that $\mvl^\vee \cong \bigwedge^{r-1} \mvl \otimes L$ and apply Riemann-Roch formula for the bundles $\mvl \otimes (S')^\vee \otimes K_C \otimes L$ and $\mvl \otimes (S')^\vee \otimes K_C \otimes  \bigwedge^{r-1} \mvl \otimes L^2$ to compute the Euler characteristic of this bundles, called $\chi_1$ and $\chi_2$ respectively. We get
    \begin{align*}
        \chi_1 & = r(-d-d_S + 2g-2) + (r_S + 1)(-d) + r(r_S +1) (d+g-1) \\ 
        \chi_2 & = r^2 (-d-d_S +2g-2) + r^2(r_S +1)(d+g-1).
    \end{align*}

    To simplify these expressions, we compute the dimensions of the first cohomology space of both vector bundles. Denote by $h_1$ the dimension \[dim_\mathbb{C} H^1 (\mvl \otimes (S')^\vee \otimes K_C \otimes L).\] From Serre duality, $h_1 = h^0 (S' \otimes \mvl^\vee \otimes L^{-1})$ and from our previous claim, we know that $h_1 = 0$ and $\chi_1 = h^0 (\mvl \otimes (S')^\vee \otimes K_C \otimes L)$. 
    For the second vector bundle, let $h_2$ denote the dimension \[dim_\mathbb{C} H^1 (\mvl \otimes (S')^\vee \otimes K_C \otimes  \bigwedge^{r-1} \mvl \otimes L^2).\] From Serre duality, $h_2 = h^0 (S' \otimes \mvl \otimes \mvl^\vee \otimes L^{-1})$. Consider the second exact row of the Butler's diagram of $(L,V)$ by $S$, twisting by $S' \otimes \mvl^\vee \otimes L^{-1}$ and taking cohomology, we obtain the exact sequence
    \[
    0 \rightarrow H^0 (\mvl \otimes S' \otimes \mvl^\vee \otimes L^{-1}) \rightarrow V \otimes H^0(S' \otimes \mvl^\vee \otimes L^{-1}) \rightarrow \cdots
    \]
    Since the second term of this sequence is zero, it follows that $h^0 (S' \otimes \mvl \otimes \mvl^\vee \otimes L^{-1})$ vanishes as well. Thus, we conclude that $h_2 = 0$ and $\chi_2 = h^0 (\mvl \otimes (S')^\vee \otimes K_C \otimes  \bigwedge^{r-1} \mvl \otimes L^2)$.
    Now, the vector space in \ref{coker1} is not zero if $(r+1) \chi_1 < \chi_2$ and this is equivalent to
    \begin{equation}
        \label{ineq1}
        r(r_S +1)d + r(d+d_S - (2g-2)) + (r_S + 1)d - r(r_S + 1)(d+g-1)>0
    \end{equation}
    
    Notice that the expression in the left side of \ref{ineq1} is greater than
    \begin{equation} \label{ineq2}
        r(d+d_S - (2g-2)) + (r_S + 1)d - r(r_S + 1)(d+g-1).
    \end{equation} 
    
    Now, we analize the last expression. Since $S\subset \mvl$ is a slope-stable maximal destabilizing subbundle, then $0\leq d_S + d$, and since we are considering $0<r<d$, we have that expression \ref{ineq2} is greater than
    \begin{equation}
        \label{ineq3}
        -(2g-2)r+r(r_S+1)+r(r_S + 1)(-d-g+1)
    \end{equation}
    Besides that, since we are considering $d \leq g-1$, we know that $-d \geq -(g-1)$ and the expression in \ref{ineq3} is greater or equal than
    \begin{equation}
        \label{ineq4}
        -(2g-2)r + r(r_S+1) + r(r_S + 1)(2-2g) = (2g-2)(-r-r(r_S + 1)) + r(r_S + 1)
    \end{equation}

    Since $r_S+1 \leq r$ then $-(r_S + 1) \geq -r$ and the expression \ref{ineq4} is greater or equal than
    \begin{equation} \label{ineq5}
        (2g-2)(-r+r^2)+r(r_S+1)
    \end{equation}
    From the hypothesis of $r \geq r_S+1>1$ and $g>2$ the expression \ref{ineq5} is always positive. Then $h^0 (S' \otimes \mvl^\vee) \neq 0$ and this implies that $h^0 (S^\vee \otimes F_{V,L}|_C) \neq 0$ which is a contradiction, then $\mvl$ has to be slope-stable.
\end{proof}

In conclusion, Theorem \ref{equivk3surf} establishes a positive answer to the Mistretta-Stoppino conjecture for generated linear series $(L,V)$ on smooth curves on $K3$ surfaces for $d\leq\text{min}\{g-1, k r\}$, that is, under these conditions we have that linear stability of a pair $(L,V)$ is equivalent to the slope-stability of the associated syzygy bundle $\mvl$. We highlight that this proposition does not contradict the counterexample presented by Castorena, Mistretta, and Torres-López in \cite[Theorem 4.7]{castorena2022linear}. Their counterexample involves a plane curve of degree 7, which lies outside the framework of the Martens theorem established in \cite[Theorem 3.1]{martens2006curves}. This theorem asserts that a complex projective K3 surface cannot contain a curve isomorphic to a smooth plane curve of degree $\geq 7$. Consequently, the counterexample provided by Castorena, Mistretta and Torres-López, which relies on a plane curve of degree 7, does not apply to the context considered in Theorem \ref{equivk3surf}. This distinction underscores the significance of the specific hypotheses and geometric contexts in which positive or negative results relating linear stability of linear series to slope stability of syzygy bundles can be derived.
\vspace{2mm}

\textbf{Final Remark.} As the reader can check, if the restriction $F_{V,L}|_C$ is slope-stable for $g-1 \leq d \leq 2g-1$, then the proof of Proposition \ref{equivk3surf} can be adapted modifying some numerical inequalities to get the slope-stability of $\mvl$.

\phantomsection
\addcontentsline{toc}{chapter}{Bibliography}

\bibliographystyle{plainurl}
\bibliography{phDThesis}

\end{document}